\documentclass[oneside,english]{amsart}
\usepackage[T1]{fontenc}
\usepackage[latin9]{inputenc}
\usepackage{units}
\usepackage{mathtools}
\usepackage{amstext}
\usepackage{amsthm}
\usepackage{amssymb}
\usepackage{stackrel}

\makeatletter
\numberwithin{equation}{section}
\numberwithin{figure}{section}
\theoremstyle{plain}
\newtheorem{thm}{\protect\theoremname}
\ifx\proof\undefined
\newenvironment{proof}[1][\protect\proofname]{\par
	\normalfont\topsep6\p@\@plus6\p@\relax
	\trivlist
	\itemindent\parindent
	\item[\hskip\labelsep\scshape #1]\ignorespaces
}{%
	\endtrivlist\@endpefalse
}
\providecommand{\proofname}{Proof}
\fi
\theoremstyle{definition}
\newtheorem{defn}[thm]{\protect\definitionname}
\theoremstyle{plain}
\newtheorem{lem}[thm]{\protect\lemmaname}
\theoremstyle{remark}
\newtheorem{rem}[thm]{\protect\remarkname}
\theoremstyle{plain}
\newtheorem{prop}[thm]{\protect\propositionname}

\usepackage{etoolbox}

\patchcmd{\@maketitle}
  {\ifx\@empty\@dedicatory}
  {\smallskip
    \begin{center}
    \footnotesize
    \begin{tabular}{c}
      Max Planck UCL Centre for Computational Psychiatry and Ageing Research, University College London, UK \\
     \end{tabular}
  \end{center}
  \ifx\@empty\@dedicatory}
  {}{}

\makeatother

\usepackage{babel}
\providecommand{\definitionname}{Definition}
\providecommand{\lemmaname}{Lemma}
\providecommand{\propositionname}{Proposition}
\providecommand{\remarkname}{Remark}
\providecommand{\theoremname}{Theorem}

\begin{document}
\title[Random vectors on a Curie-Weiss model with random couplings]{Random vectors on the spin configuration of a Curie-Weiss model on
Erd\H{o}s-Rényi random graphs}
\author{Dominik R. Bach}
\address{\textsuperscript{}Max Planck UCL Centre for Computational Psychiatry
and Ageing Research, University College London, United Kingdom}
\thanks{The author would like to thank Prof. Werner Kirsch, FernUniversität
in Hagen, Germany, for critical support and guidance through this
work.}
\date{11 January 2024}
\begin{abstract}
This article is concerned with the asymptotic behaviour of random
vectors in a diluted ferromagnetic model. We consider a model introduced
by Bovier \& Gayrard (1993) with ferromagnetic interactions on a directed
Erd\H{o}s-Rényi random graph. Here, directed connections between graph
nodes are uniformly drawn at random with a probability $p$ that depends
on the number of nodes $N$ and is allowed to go to zero in the limit.
If $Np\longrightarrow\infty$ in this model, Bovier \& Gayrard (1993)
proved a law of large numbers almost surely, and Kabluchko et al.
(2020) proved central limit theorems in probability. Here, we generalise
these results for $\beta<1$ in the regime $Np\longrightarrow\infty$.
We show that all those random vectors on the spin configuration that
have a limiting distribution under the Curie-Weiss model converge
weakly towards the same distribution under the diluted model, in probability
on graph realisations. This generalises various results from the Curie-Weiss
model to the diluted model. As a special case, we derive a law of
large numbers and central limit theorem for two disjoint groups of
spins. 
\end{abstract}

\keywords{Ising model, dilute Curie-Weiss model, Law of Large numbers, Central
Limit Theorem, random graphs}

\maketitle
\newpage{}

\section{Introduction }

\subsection{Background}

Diluted ferromagnetic models might be taken to describe the behaviour
of a quenched alloy of ferromagnetic and a non-magnetic material \cite{georgii1981spontaneous,griffiths1968random},
or the behaviour of voters who interact at random \cite{Kirsch_2020}.
Even though more realistic models have been suggested for the latter
case \cite{LOWE2020124735}, the diluted ferromagnetic model remains
an abstract and tractable approach. The study of diluted ferromagnetic
models has historically focused on Ising type models (e.g. \cite{chayes1985low,dunn1975scaling,griffiths1968random,georgii1981spontaneous}).
Bovier \& Gayrard (1993) \cite{bovier1993thermodynamics} introduced
a diluted version of the Curie-Weiss model on directed Erd\H{o}s-Rényi
random graphs. They proved a law of large numbers almost surely. For
a very similar model, recent work proved a central limit theorem in
probability over graph configurations, for the case where the external
magnetic field $h=0$ and $\beta<1$, provided that $N^{2}p^{3}\longrightarrow\infty$
\cite{kabluchko2019fluctuations}. In \cite{Kabluchko_2020}, they
sharpened the approximations to the less tight regime $Np\longrightarrow\infty$
for $\beta<1$, and derived a central limit theorem for $\beta=1$.
In \cite{Kabluchko_low_temp}, they derived (conditional) central
limit theorems for the the cases $\beta>0$ and external magnetic
field $h>0$. Some physical properties of the model have been analysed
in \cite{deSanctis2008,Agliari2008}, and fluctuations of the partition
function over graph configurations in \cite{Kabluchko2020fluctuations}.
The goal of this note is to generalise the central limit theorem for
$Np\longrightarrow\infty$ and $\beta<1$ to a wide set of random
vectors, including a law of large numbers and central limit theorem
for the homogenous two-group case \cite{Kirsch_2020}. 

\subsection{Description of the model}

We consider spin configurations on a sequence of directed Erd\H{o}s-Rényi
random graphs with $N=1,2,...$ nodes. Each node $i\in\mathbb{N}_{N}$
can take spin $x_{i}\in\left\{ -1,+1\right\} $ such that $\mathbf{x}_{N}\coloneqq\left(x_{1},x_{2},...,x_{N}\right)\in\mathcal{X}_{N}\coloneqq\left\{ -1,+1\right\} ^{N}$.
The presence of a directed edge from node $i$ to node $j$ in graph
$N$ is denoted by the indicator variable $\varepsilon_{N,i,j}\in\left\{ 0,1\right\} $.
Let $p:\mathbb{N}\rightarrow\left]0,1\right]$ be an arbitrary function.
We define a product probability space over the entire graph sequence:

\begin{equation}
\left(\Omega_{\varepsilon},\mathcal{A}_{\varepsilon},\mathbb{P_{\varepsilon}}\right)\coloneqq\left(\underset{\left(N,i,j\right)\in\mathbb{N}^{3}}{\bigtimes}\left\{ 0,1\right\} ,\underset{\left(N,i,j\right)\in\mathbb{N}^{3}}{\bigotimes}\mathcal{P}\left(\left\{ 0,1\right\} \right),\underset{\left(N,i,j\right)\in\mathbb{N}^{3}}{\bigotimes}\mathbb{P}_{N,i,j}\right)
\end{equation}
with $\mathbb{P}_{N,i,j}\left(\varepsilon_{N,i,j}=1\right)\coloneqq p\left(N\right)$.
For every fixed graph realisation, we define two probability measures
over $\left(\mathcal{X}_{N},\mathcal{P}\left(\mathcal{X}_{N}\right)\right)$,
corresponding to the well-known Curie-Weiss model, and the Bovier-Gayrard
model, with:

\begin{equation}
\mathbb{P}_{N}^{\left(CW\right)}\left(\mathbf{x}\right)\coloneqq\mu_{N}^{\left(CW\right)}\left(\mathbf{x}\right)/Z_{N}^{(CW)}\coloneqq e^{\frac{\beta}{2N}\stackrel[i,j]{N}{\sum}x_{i}x_{j}}/\underset{\mathbf{x}\in\mathcal{X}_{N}}{\sum}\mu_{N}^{\left(CW\right)}\left(\mathbf{x}\right).
\end{equation}

\begin{equation}
\mathbb{P}_{N}^{\left(BG\right)}\left(\mathbf{x}\right)\coloneqq\mu_{N}^{\left(BG\right)}\left(\mathbf{x}\right)/Z_{N}^{(BG)}\coloneqq e^{\frac{\beta}{2Np}\stackrel[i,j]{N}{\sum}\varepsilon_{N,i,j}x_{i}x_{j}}/\underset{\mathbf{x}\in\mathcal{X}_{N}}{\sum}\mu_{N}^{\left(BG\right)}\left(\mathbf{x}\right).
\end{equation}

We denote the total magnetisation (sum of all spins) of the graph
with $s_{N}\left(\mathbf{x}_{N}\right)$; recall that $s_{N}^{2}\left(\mathbf{x}_{N}\right)=\stackrel[i,j]{N}{\sum}x_{i}x_{j}$.
We use $\mathbf{x}$, $\mathbf{x}_{1}$, etc. when the number of elements
$N$ is clear from the context. For the remainder of the paper, we
assume $\beta<1$ and $Np\longrightarrow\infty$ as $N\longrightarrow\infty$.
Furthermore, we define $\left(f_{N}\right)_{N\in\mathbb{N}}$, $f_{N}:\mathcal{X}_{N}\longrightarrow\mathbb{R}$
with $\left|f_{N}\right|<M\in\mathbb{R}$ for all $N\in\mathbb{N}$.
We denote the positive and negative parts of $f_{N}$ with $f_{N}^{+}>0$
and $f_{N}^{-}>0$. 

\section{Main results}

Denote weak convergence with $\Longrightarrow$, and $\mathbb{P}_{\varepsilon}$-stochastic
convergence with $\overset{\mathcal{\mathbb{P}_{\varepsilon}}}{\longrightarrow}$.
For two probability measures $\mathbb{P}_{1}$ and $\mathbb{P}_{2}$,
denote the Levy metric of weak convergence with $d_{L}\left(\mathbb{P}_{1},\mathbb{P}_{2}\right)$.
Recall that $\left(\mathbb{P}_{n}\Longrightarrow\mathbb{P}\right)\Leftrightarrow\left(d_{L}\left(\mathbb{P}_{n},\mathbb{P}\right)\longrightarrow0\right)$
(e.g. \cite{Elstrodt_2011}).
\begin{thm}
\label{thm:weak-convergence-probability-Rd}Let $\beta<1$ and $Np\longrightarrow\infty$
as $N\longrightarrow\infty$, let $\left(Y_{N}\right)_{N\in\mathbb{N}}$
with $Y_{N}:\mathcal{X}_{N}\longrightarrow\mathbb{R}^{d}$, $d\in\mathbb{\mathbb{N}}$,
have limiting image distribution $\mathbb{P}_{Y}^{\left(CW\right)}$
such that $\left(\mathbb{P}_{N}^{\left(CW\right)}\circ Y_{N}^{-1}\right)\Longrightarrow\mathbb{P}_{Y}^{\left(CW\right)}$.
Then 

\[
d_{L}\left(\mathbb{P}_{N}^{\left(BG\right)}\circ Y_{N}^{-1},\mathbb{P}_{Y}^{\left(CW\right)}\right)\overset{\mathcal{\mathbb{P}_{\varepsilon}}}{\longrightarrow}0.
\]

In shorthand notation, we write:

\[
Y_{N}\stackrel[BG]{\mathbb{P}_{\varepsilon}}{\Longrightarrow}\mathbb{P}_{Y}^{\left(CW\right)}.
\]
\end{thm}
\begin{thm}
Let $\beta<1$ and $Np\longrightarrow\infty$ as $N\longrightarrow\infty$,
let $s_{N}^{\left(1\right)}\left(\mathbf{x}_{N}\right)$ and $s_{N}^{\left(2\right)}\left(\mathbf{x}_{N}\right)$
be the respective sums of two disjoint subsets of spins with respective
cardinality $N_{N,1}$ and $N_{N,2}$. 

(1) Law of large numbers. Let $m\left(\beta\right)$ be the unique
positive solution to $x=\tanh\left(\beta x\right)$. Then
\[
\left(\frac{1}{N_{N,1}}s_{N}^{\left(1\right)}\left(\mathbf{x}_{N}\right),\frac{1}{N_{N,2}}s_{N}^{\left(2\right)}\left(\mathbf{x}_{N}\right)\right)\stackrel[BG]{\mathbb{P}_{\varepsilon}}{\Longrightarrow}\frac{1}{2}\left(\delta_{\left(-m\left(\beta\right),-m\left(\beta\right)\right)}+\delta_{\left(+m\left(\beta\right),+m\left(\beta\right)\right)}\right),
\]

(2) Central limit theorem. Assume existence of 

\[
\alpha_{1}\coloneqq\underset{N\rightarrow\infty}{\lim}\frac{N_{N,1}}{N},\qquad\alpha_{2}\coloneqq\underset{N\rightarrow\infty}{\lim}\frac{N_{N,2}}{N},
\]

and define 

\[
C\coloneqq\left[\begin{array}{cc}
1+\alpha_{1}\frac{\beta}{1-\beta} & \sqrt{\alpha_{1}\alpha_{2}}\frac{\beta}{1-\beta}\\
\sqrt{\alpha_{1}\alpha_{2}}\frac{\beta}{1-\beta} & 1+\alpha_{2}\frac{\beta}{1-\beta}
\end{array}\right],
\]

then:

\[
\left(\frac{1}{\sqrt{N_{N,1}}}s_{N}^{\left(1\right)}\left(\mathbf{x}_{N}\right),\frac{1}{\sqrt{N_{N,2}}}s_{N}^{\left(2\right)}\left(\mathbf{x}_{N}\right)\right)\stackrel[BG]{\mathbb{P}_{\varepsilon}}{\Longrightarrow}\xi,\;\xi\thicksim\mathcal{N}\left(\left(0,0\right),C\right).
\]
\end{thm}
\begin{proof}
This theorem follows directly from theorem \ref{thm:weak-convergence-probability-Rd}
and the results for the Curie-Weiss model proven in \cite{Kirsch_2020}.
\end{proof}

\section{Technical preparation}

\subsection{Results from previous work}

The following definition and lemma are adapted from \cite{Kabluchko_2020}.
In this reference, $m$ is fixed at $m=1/5$, but it is easy to show
that the results hold for any fixed $0<m<1$. Part (1) restates Lemmata
3.1 - 3.2 in \cite{Kabluchko_2020}. Part (2) is implicit in their
proofs and made explicit here for clarity. For a random variable $G:\Omega_{\varepsilon}\rightarrow\mathbb{R}$
and fixed $N$, we define its expectation over graph configurations
with $E_{\varepsilon,N}\left(G\right)$.
\begin{defn}
\label{def:X_T}For fixed $0<m<1$, define the following sets of ``typical''
spin configurations, and pairs of spin configurations:

\[
\mathcal{X}_{T,N}\coloneqq\left\{ \mathbf{x}\in\mathcal{X}_{N}:s_{N}^{2}\left(\mathbf{x}\right)\leq N\left(Np\right)^{m}\right\} ,
\]

\[
\mathcal{X}_{T,N}^{\left(2\right)}\coloneqq\left\{ \left(\mathbf{x}_{1},\mathbf{x}_{2}\right)\in\mathcal{X}_{N}^{2}:s_{N}^{2}\left(\mathbf{x}_{1}\right)\leq N\left(Np\right)^{m},s_{N}^{2}\left(\mathbf{x}_{2}\right)\leq N\left(Np\right)^{m},\left(\stackrel[i=1]{N}{\sum}x_{1,i}x_{2,i}\right)^{2}\leq N\left(Np\right)^{m}\right\} .
\]
\end{defn}
\begin{lem}
\label{thm:Kabluchko-beta-1}

(1) For fixed $\mathbf{x}_{N}\in\mathcal{X}_{T,N}$ and $\left(\mathbf{x}_{N,1},\mathbf{x}_{N,2}\right)\in\mathcal{X}_{T,N}^{\left(2\right)}$:

\[
E_{\varepsilon,N}\left(\frac{\mu_{N}^{\left(BG\right)}\left(\mathbf{x}_{N}\right)}{\cosh\left(\frac{\beta}{2Np}\right)\exp\left(\stackrel[i,j]{N}{\sum}\varepsilon_{N,i,j}\right)}\right)=\exp\left(-\frac{\beta^{2}}{8}+\frac{\beta}{2N}s_{N}^{2}\left(\mathbf{x}\right)+c_{N,1}\right),
\]

\[
E_{\varepsilon,N}\left(\frac{\mu_{N}^{\left(BG\right)}\left(\mathbf{x}_{N,1}\right)}{\cosh\left(\frac{\beta}{2Np}\right)\exp\left(\stackrel[i,j]{N}{\sum}\varepsilon_{N,i,j}\right)}\frac{\mu_{N}^{\left(BG\right)}\left(\mathbf{x}{}_{N,2}\right)}{\cosh\left(\frac{\beta}{2Np}\right)\exp\left(\stackrel[i,j]{N}{\sum}\varepsilon_{N,i,j}\right)}\right)
\]
\[
=\exp\left(-\frac{\beta^{2}}{4}+\frac{\beta}{2N}\left(s_{N}^{2}\left(\mathbf{x}_{N,1}\right)+s_{N}^{2}\left(\mathbf{x}_{N,2}\right)\right)+c_{N,2}\right),
\]

where \textup{$\left(c_{N,1}\right)_{N\in\mathbb{N}}$} and $\left(c_{N,2}\right)_{N\in\mathbb{N}}$
are null sequences of real numbers that do not depend on $\mathbf{x}_{N},\mathbf{x}_{N,1},\mathbf{x}_{N,2}$. 

(2) For fixed $\mathbf{x}_{N},\mathbf{x}_{N,1},\mathbf{x}_{N,2}\in\mathcal{X}_{N}$:

\[
E_{\varepsilon,N}\left(\frac{\mu_{N}^{\left(BG\right)}\left(\mathbf{x}_{N}\right)}{\cosh\left(\frac{\beta}{2Np}\right)\exp\left(\stackrel[i,j]{N}{\sum}\varepsilon_{N,i,j}\right)}\right)=\exp\left(\frac{\beta^{2}}{8}+\frac{\beta}{2N}s_{N}^{2}\left(\mathbf{x}_{N}\right)\left(1+o\left(1\right)\right)+o\left(1\right)\right)
\]

\[
E_{\varepsilon,N}\left(\frac{\mu_{N}^{\left(BG\right)}\left(\mathbf{x}_{N,1}\right)}{\cosh\left(\frac{\beta}{2Np}\right)\exp\left(\stackrel[i,j]{N}{\sum}\varepsilon_{N,i,j}\right)}\frac{\mu_{N}^{\left(BG\right)}\left(\mathbf{x}{}_{N,2}\right)}{\cosh\left(\frac{\beta}{2Np}\right)\exp\left(\stackrel[i,j]{N}{\sum}\varepsilon_{N,i,j}\right)}\right)
\]

\[
=\exp\left(-\frac{\beta^{2}}{4}+o\left(1\right)+\left(\frac{\beta}{2N}+o\left(1\right)\right)\left(s_{N}^{2}\left(\mathbf{x}_{N,1}\right)+s_{N}^{2}\left(\mathbf{x}_{N,2}\right)\right)+\frac{o\left(1\right)}{2N}\left(\stackrel[i]{N}{\sum}x_{i,1}x_{i,2}\right)^{2}\right).
\]
\end{lem}
The following lemma is proven as corollary 4.15 in \cite{Elstrodt_2011}: 
\begin{lem}
\label{prop:stochastic_convergence_sure_subsequences}Let $\left(\Omega,\mathcal{A},\mathbb{P}\right)$
be a probability space, let $Y_{n},Y:\left(\Omega,\mathcal{A}\right)\rightarrow\left(\mathbb{R},\mathcal{B}\right)$
be real-valued random variables for all $n\in\mathbb{N}$. Then the
following two statements are equivalent:

1. $Y_{n}\overset{\mathcal{\mathbb{P}}}{\longrightarrow}Y$.

2. Every subsequence of $\left(Y_{n}\right)_{n\in\mathbb{N}}$ has
a subsequence that converges to $Y$ almost surely.
\end{lem}

\subsection{Technical preparation for the proofs of proposition \ref{lem:moment convergence-Np} }

The following lemmata and proofs generalise the approach taken in
\cite{kabluchko2019fluctuations,Kabluchko_2020}. 
\begin{defn}
$S_{N}\coloneqq\left\{ -N,-N+2,..,N-2,N\right\} $.
\end{defn}
\begin{lem}
\label{lem:exp-delta-1D-bounded} Let $\delta\in\mathbb{R}^{+}$.
Then:

\textup{
\begin{equation}
\underset{\mathbf{x}\in\left(\mathcal{X}_{T,N}\right)^{c}}{\sum}f_{N}^{+}\left(\mathbf{x}\right)e^{\frac{1-\delta}{2N}s_{N}^{2}}=2^{N}o\left(1\right).
\end{equation}
}
\end{lem}
\begin{proof}
For fixed $s\in S_{N}$, let $\nu_{N,s}\coloneqq\left|\left\{ \mathbf{x}\in\mathcal{X}_{N}:s_{N}=s\right\} \right|$,
i.e. the number of spin configurations with $\frac{s+N}{2}$ positive
spins. By the de Moivre-Laplace local limit theorem: 

\begin{equation}
\nu_{N,s}=\left(\begin{array}{c}
N\\
\frac{s+N}{2}
\end{array}\right)=2^{N}\frac{\left(1+o\left(1\right)\right)}{\sqrt{\frac{1}{2}\pi N}}\exp\left(-\frac{s^{2}}{2N}\right).\label{eq:bin_approx}
\end{equation}

Recalling that $f_{N}^{+}\left(\mathbf{x}\right)<M$, we have:
\begin{equation}
\underset{\mathbf{x}\in\left(\mathcal{X}_{T,N}\right)^{c}}{\sum}f_{N}^{+}\left(\mathbf{x}\right)e^{\frac{1-\delta}{2N}s_{N}^{2}}
\end{equation}

\begin{equation}
\leq2^{N+1}\frac{\left(1+o\left(1\right)\right)}{\sqrt{\frac{1}{2}\pi N}}M\underset{\left\{ s\in S_{N}:\sqrt{N\left(Np\right)^{m}}\leq s\leq N\right\} }{\sum}e^{-\delta\frac{s^{2}}{2N}}
\end{equation}

(by eq. (\ref{eq:bin_approx}))

\begin{equation}
\leq2^{N+1}\frac{M}{\sqrt{N}}\underset{\left\{ s\in S_{N}:\sqrt{N\left(Np\right)^{m}}\leq s\leq N\right\} }{\sum}e^{-\delta\frac{s^{2}}{2N}},\label{eq:sum2int-first-expansion}
\end{equation}

for large enough $N$. Now for $t\in\left[s/\sqrt{N},\left(s+2\right)/\sqrt{N}\right[$
:

\begin{equation}
e^{-\delta\frac{\left(s+2\right)^{2}}{2N}}\leq e^{-\delta\frac{t^{2}}{2}},
\end{equation}

and so

\begin{equation}
\frac{2}{\sqrt{N}}e^{-\delta\frac{\left(s+2\right)^{2}}{2N}}\leq\int_{s/\sqrt{N}}^{\left(s+2\right)/\sqrt{N}}e^{-\delta\frac{t^{2}}{2}}dt.\label{eq:sum2int-exp-only}
\end{equation}

Hence, continuing from eq. (\ref{eq:sum2int-first-expansion}):

\begin{equation}
2^{N+1}\frac{M}{\sqrt{N}}\underset{\left\{ s\in S_{N}:\sqrt{N\left(Np\right)^{m}}\leq s\leq N\right\} }{\sum}e^{-\delta\frac{s^{2}}{2N}}
\end{equation}

\begin{equation}
\leq2^{N}M\int_{\sqrt{\left(Np\right)^{m}}}^{\left(N+2\right)/\sqrt{N}}e^{-\delta\frac{t^{2}}{2}}dt
\end{equation}

\begin{equation}
\leq2^{N}M\int_{\sqrt{\left(Np\right)^{m}}}^{\infty}e^{-\delta\frac{t^{2}}{2}}dt\label{eq:expand-range-int-to-inf}
\end{equation}

\begin{equation}
=2^{N}o\left(1\right).
\end{equation}
\end{proof}
\begin{lem}
\label{lem:exp-delta-3D-bounded} Let $\delta\in\mathbb{R}^{+}$.
Then: 

\textup{
\begin{equation}
\underset{\mathbf{\left(x_{1},x_{2}\right)}\in\left(\mathcal{X}_{T,N}^{\left(2\right)}\right)^{c}}{\sum}f_{N}^{+}\left(\mathbf{x_{1}}\right)f_{N}^{+}\left(\mathbf{x_{2}}\right)e^{\frac{1-\delta}{2N}\left(s_{N}^{2}\left(\mathbf{x}_{1}\right)+s_{N}^{2}\left(\mathbf{x}_{2}\right)+\left(\stackrel[i=1]{N}{\sum}x_{i,1}x_{i,2}\right)^{2}\right)}\leq\left(2^{N}o\left(1\right)\right)^{2}.
\end{equation}
}
\end{lem}
\begin{proof}
\label{def:V_N-nu_N}For $s,t,u\in\mathbb{\mathbb{Z}}_{N}$, define
$\nu_{N}\left(s,t,u\right)\coloneqq\left|\left\{ \mathbf{\left(x_{1},x_{2}\right)}\in\mathcal{X}_{N}^{2}:s_{N}\left(\mathbf{x}_{1}\right)=s,s_{N}\left(\mathbf{x}_{2}\right)=t,\stackrel[i=1]{N}{\sum}x_{i,1}x_{i,2}=u\right\} \right|.$
Then there exists a constant $C\in\mathbb{R}$ such that $\nu_{N}\left(s,t,u\right)<2^{2N}\frac{C}{N^{3/2}}\exp\left(-\frac{s^{2}}{2N}-\frac{t^{2}}{2N}-\frac{u^{2}}{2N}\right)$
(proof in \cite{Kabluchko_2020}). Now let $W_{N}\coloneqq\left\{ \left(s,t,u\right)\in S_{N}^{3}:s\geq\sqrt{N\left(Np\right)^{m}}\vee t\geq\sqrt{N\left(Np\right)^{m}}\vee u\geq\sqrt{N\left(Np\right)^{m}}\right\} .$
Then:
\begin{equation}
\underset{\mathbf{\left(x_{1},x_{2}\right)}\in\left(\mathcal{X}_{T,N}^{\left(2\right)}\right)^{c}}{\sum}f_{N}^{+}\left(\mathbf{x_{1}}\right)f_{N}^{+}\left(\mathbf{x_{2}}\right)e^{\frac{1-\delta}{2N}\left(s_{N}^{2}\left(\mathbf{x}_{1}\right)+s_{N}^{2}\left(\mathbf{x}_{2}\right)+\left(\stackrel[i=1]{N}{\sum}x_{i,1}x_{i,2}\right)^{2}\right)}
\end{equation}

\begin{equation}
\leq2M^{2}\underset{\left(s,t,u\right)\in W_{N}}{\sum}\nu_{N}\left(s,t,u\right)e^{\frac{1-\delta}{2N}\left(s^{2}+t^{2}+u^{2}\right)}
\end{equation}

\begin{equation}
<M^{2}2^{2N+1}\frac{C}{N^{3/2}}\underset{\left(s,t,u\right)\in W_{N}}{\sum}e^{-\frac{\delta}{2N}\left(s^{2}+t^{2}+u^{2}\right)}\label{eq:3D-approx-first-expansion}
\end{equation}

\begin{equation}
=\left(2^{N}o\left(1\right)\right)^{2}
\end{equation}

(using eq. (\ref{eq:sum2int-exp-only}) for $s$, $t$, and $u$ and
expanding the range of integration as in eq. (\ref{eq:expand-range-int-to-inf})).
\end{proof}
\begin{lem}
\label{lem:Z(CW)-order}There exists some $C\in\mathbb{R}^{+}$ independent
of $N$ such that

\[
Z_{N}^{\left(CW\right)}\geq2^{N}C.
\]
\end{lem}
\begin{proof}
We have:
\begin{equation}
Z_{N}^{\left(CW\right)}=\underset{\mathbf{x}\in\mathcal{X}_{N}}{\sum}e^{\frac{\beta}{2N}s_{N}^{2}\left(\mathbf{x}\right)}
\end{equation}

\begin{equation}
\geq\frac{2^{N}}{\sqrt{N}}\underset{\left\{ s\in S_{N}:0\leq s\leq N\right\} }{\sum}e^{\left(\beta-1\right)\frac{s^{2}}{2N}},\label{eq:Z_first_expansion}
\end{equation}

for large enough $N$, using eq. (\ref{eq:bin_approx}). Because $\beta-1<0$,
we have, for $t\in\left[s/\sqrt{N},\left(s+2\right)/\sqrt{N}\right[$
:

\begin{equation}
e^{\left(\beta-1\right)\frac{s^{2}}{2N}}\geq e^{\left(\beta-1\right)\frac{t^{2}}{2}},
\end{equation}

and so

\begin{equation}
\frac{2}{\sqrt{N}}e^{\left(\beta-1\right)\frac{s^{2}}{2N}}\geq\int_{s/\sqrt{N}}^{\left(s+2\right)/\sqrt{N}}e^{\left(\beta-1\right)\frac{t^{2}}{2}}dt.
\end{equation}

Hence, continuing from eq. (\ref{eq:Z_first_expansion}):

\begin{equation}
\frac{2^{N}}{\sqrt{N}}\underset{\left\{ s\in S_{N}:0\leq s\leq N\right\} }{\sum}e^{\left(\beta-1\right)\frac{s^{2}}{2N}}
\end{equation}

\begin{equation}
\geq2^{N-1}\int_{0}^{\left(N+2\right)/\sqrt{N}}e^{\left(\beta-1\right)\frac{t^{2}}{2}}dt
\end{equation}

\begin{equation}
\geq2^{N}\frac{1}{4}\sqrt{\frac{\pi}{2\left(1-\beta\right)}},
\end{equation}

for large enough $N$, since $erf\left(x\right)\longrightarrow1>\nicefrac{1}{2}$
for $x\longrightarrow\infty$. 
\end{proof}

\section{Proof of the main results}

\subsection{Convergence of integral quotients}

\subsubsection{Definitions and propositions}
\begin{defn}
\label{def:functionals}For ease of notation, we define
\begin{equation}
E_{N}^{\left(\mu,BG\right)}\left(Y_{N}\right)\coloneqq\underset{\mathbf{x}\in\mathcal{X}_{N}}{\sum}Y_{N}\left(\mathbf{x}\right)\mu_{N}^{\left(BG\right)}\left(\mathbf{x}\right),
\end{equation}

\begin{equation}
E_{N}^{\left(P,BG\right)}\left(Y_{N}\right)\coloneqq\underset{\mathbf{x}\in\mathcal{X}_{N}}{\sum}Y_{N}\left(\mathbf{x}\right)\mathbb{P}{}_{N}^{\left(BG\right)}\left(\mathbf{x}\right)=E_{N}^{\left(\mu,BG\right)}\left(Y_{N}\right)/E_{N}^{\left(\mu,BG\right)}\left(1\right),
\end{equation}

and similarly for $E_{N}^{\left(\mu,CW\right)}\left(Y_{N}\right)$
and $E_{N}^{\left(P,CW\right)}\left(Y_{N}\right)$. Furthermore, we
define

\[
R_{N}\left(f_{N}^{+}\right)\coloneqq\begin{cases}
\frac{E_{N}^{\left(\mu,BG\right)}\left(f_{N}^{+}\right)}{\cosh\left(\frac{\beta}{2Np}\right)\exp\left(-\frac{\beta^{2}}{8}+\stackrel[i,j]{N}{\sum}\varepsilon_{N,i,j}\right)E_{N}^{\left(\mu,CW\right)}\left(f_{N}^{+}\right)} & \left(E_{N}^{\left(\mu,CW\right)}\left(f_{N}^{+}\right)\neq0\right)\\
1 & \left(E_{N}^{\left(\mu,CW\right)}\left(f_{N}^{+}\right)=0\right)
\end{cases}
\]

\[
T_{N}\left(f_{N}^{+}\right)\coloneqq\frac{E_{N}^{\left(\mu,BG\right)}\left(f_{N}^{+}\right)}{\cosh\left(\frac{\beta}{2Np}\right)\exp\left(-\frac{\beta^{2}}{8}+\stackrel[i,j]{N}{\sum}\varepsilon_{N,i,j}\right)Z_{N}^{\left(\mu,CW\right)}}
\]
\end{defn}
\begin{rem}
$E_{N}^{\left(\mu,CW\right)}\left(f_{N}^{+}\right)=0$ if and only
if $f_{N}^{+}$ vanishes everywhere on $\mathcal{X}_{N}$, in which
case, $E_{N}^{\left(\mu,BG\right)}\left(f_{N}^{+}\right)=E_{N}^{\left(\mu,CW\right)}\left(f_{N}^{+}\right)=0$. 
\end{rem}
\begin{prop}
\label{lem:moment convergence-Np}\vphantom{}

(1) If $\underset{N\rightarrow\infty}{\lim}E_{N}^{\left(P,CW\right)}\left(f_{N}^{+}\right)>0$
exists, then $R_{N}\left(f_{N}^{+}\right)\overset{\mathcal{\mathbb{P}_{\varepsilon}}}{\longrightarrow}1.$

(2) If $\underset{N\rightarrow\infty}{\lim}E_{N}^{\left(P,CW\right)}\left(f_{N}^{+}\right)=0$,
then $T_{N}\left(f_{N}^{+}\right)\overset{\mathcal{\mathbb{P}_{\varepsilon}}}{\longrightarrow}0.$
\end{prop}

\subsubsection{Preparation: expanding the expectations}
\begin{lem}
\label{lem:exp-2}If $\underset{N\rightarrow\infty}{\lim}E_{N}^{\left(P,CW\right)}\left(f_{N}^{+}\right)>0$
exists, then there exists a sequence $\left(a_{N,1}\right)_{N\in\mathbb{N}}$
with $a_{N,1}\longrightarrow1$ as $Np\longrightarrow\infty$, such
that $E_{\varepsilon,N}\left(R_{N}\left(f_{N}^{+}\right)\right)\geq a_{N,1}.$
\end{lem}
\begin{proof}
We split into typical and atypical spin configurations for some fixed
$0<m<1$:

\begin{equation}
E_{\varepsilon,N}\left(\frac{E_{N}^{\left(\mu,BG\right)}\left(f_{N}^{+}\right)}{\cosh\left(\frac{\beta}{2Np}\right)\exp\left(\stackrel[i,j]{N}{\sum}\varepsilon_{N,i,j}\right)}\right)=E_{\varepsilon,N}\left(\underset{\mathbf{x}\in\mathcal{X}_{N}}{\sum}f_{N}^{+}\left(\mathbf{x}\right)\frac{\mu_{N}^{\left(BG\right)}\left(\mathbf{x}\right)}{\cosh\left(\frac{\beta}{2Np}\right)\exp\left(\stackrel[i,j]{N}{\sum}\varepsilon_{N,i,j}\right)}\right)
\end{equation}

\begin{equation}
\geq\underset{\mathbf{x}\in\mathcal{X}_{T,N}}{\sum}f_{N}^{+}\left(\mathbf{x}\right)E_{\varepsilon,N}\left(\frac{\mu_{N}^{\left(BG\right)}\left(\mathbf{x}\right)}{\cosh\left(\frac{\beta}{2Np}\right)\exp\left(\stackrel[i,j]{N}{\sum}\varepsilon_{N,i,j}\right)}\right)
\end{equation}

(because all terms are non-negative)

\begin{equation}
=\underset{\mathbf{x}\in\mathcal{X}_{T,N}}{\sum}f_{N}^{+}\left(\mathbf{x}\right)\exp\left(-\frac{\beta^{2}}{8}+\frac{\beta}{2N}s_{N}^{2}\left(\mathbf{x}\right)+c_{N,1}\right),
\end{equation}

(by lemma \ref{thm:Kabluchko-beta-1})

\begin{equation}
\geq\exp\left(-\left|c_{N,1}\right|\right)e^{-\frac{\beta^{2}}{8}}\underset{\mathbf{x}\in\mathcal{X}_{T,N}}{\sum}f_{N}^{+}\left(\mathbf{x}\right)e^{\frac{\beta}{2N}s_{N}^{2}\left(\mathbf{x}\right)}
\end{equation}

\begin{equation}
=\exp\left(-\left|c_{N,1}\right|\right)e^{-\frac{\beta^{2}}{8}}E_{N}^{\left(\mu,CW\right)}\left(f_{N}^{+}\right)\left(1-\frac{\underset{\mathbf{x}\in\mathcal{X}_{T,N}^{c}}{\sum}f_{N}^{+}\left(\mathbf{x}\right)e^{\frac{\beta}{2N}s_{N}^{2}}}{E_{N}^{\left(P,CW\right)}\left(f_{N}^{+}\right)Z_{N}^{\left(CW\right)}}\right)
\end{equation}

\begin{equation}
\geq\exp\left(-\left|c_{N,1}\right|\right)e^{-\frac{\beta^{2}}{8}}E_{N}^{\left(\mu,CW\right)}\left(f_{N}^{+}\right)\left(1-o\left(1\right)\right),
\end{equation}

(for large enough $N$, using proof assumptions, lemma \ref{lem:exp-delta-1D-bounded}
with $\left(1-\delta\right)=\beta<1$, and lemma \ref{lem:Z(CW)-order})

\begin{equation}
=a_{N,1}e^{-\frac{\beta^{2}}{8}}E_{N}^{\left(\mu,CW\right)}\left(f_{N}^{+}\right),
\end{equation}

\begin{equation}
a_{N,1}\coloneqq\exp\left(-\left|c_{N,1}\right|\right)\left(1-o\left(1\right)\right)\longrightarrow1,\quad Np\longrightarrow\infty.
\end{equation}

The lemma follows with definition \ref{def:functionals}.
\end{proof}
\begin{lem}
\label{lem:sq-exp-2}If $\underset{N\rightarrow\infty}{\lim}E_{N}^{\left(P,CW\right)}\left(f_{N}^{+}\right)>0$
exists, then there exists a sequence $\left(a_{N,2}\right)_{N\in\mathbb{N}}$
with $a_{N,2}\longrightarrow1$ as $Np\longrightarrow\infty$, such
that $E_{\varepsilon,N}\left(R_{N}^{2}\left(f_{N}^{+}\right)\right)\leq a_{N,2}.$
\end{lem}
\begin{proof}
\begin{equation}
E_{\varepsilon,N}\left(\left(\frac{E_{N}^{\left(\mu,BG\right)}\left(f_{N}^{+}\right)}{\cosh\left(\frac{\beta}{2Np}\right)\exp\left(\stackrel[i,j]{N}{\sum}\varepsilon_{N,i,j}\right)}\right)^{2}\right)=E_{\varepsilon,N}\left(\left(\underset{\mathbf{x}\in\mathcal{X}_{N}}{\sum}f_{N}^{+}\left(\mathbf{x}\right)\frac{\mu^{\left(BG\right)}\left(\mathbf{x}\right)}{\cosh\left(\frac{\beta}{2Np}\right)\exp\left(\stackrel[i,j]{N}{\sum}\varepsilon_{N,i,j}\right)}\right)^{2}\right)
\end{equation}

\begin{equation}
=\underset{\mathbf{\left(x_{1},x_{2}\right)}\in\mathcal{X}_{N}^{2}}{\sum}f_{N}^{+}\left(\mathbf{x_{1}}\right)f_{N}^{+}\left(\mathbf{x_{2}}\right)E_{\varepsilon,N}\left(\frac{\mu^{\left(BG\right)}\left(\mathbf{x_{1}}\right)}{\cosh\left(\frac{\beta}{2Np}\right)\exp\left(\stackrel[i,j]{N}{\sum}\varepsilon_{N,i,j}\right)}\frac{\mu^{\left(BG\right)}\left(\mathbf{x_{2}}\right)}{\cosh\left(\frac{\beta}{2Np}\right)\exp\left(\stackrel[i,j]{N}{\sum}\varepsilon_{N,i,j}\right)}\right)
\end{equation}

\[
=\underset{\mathbf{\left(x_{1},x_{2}\right)}\in\mathcal{X}_{T,N}^{\left(2\right)}}{\sum}f_{N}^{+}\left(\mathbf{x_{1}}\right)f_{N}^{+}\left(\mathbf{x_{2}}\right)\exp\left(-\frac{\beta^{2}}{4}+\frac{\beta}{2N}\left(s_{N}^{2}\left(\mathbf{x}_{1}\right)+s_{N}^{2}\left(\mathbf{x}_{2}\right)\right)+c_{N,2}\right)
\]

\[
+\underset{\mathbf{\left(x_{1},x_{2}\right)}\in\left(\mathcal{X}_{T,N}^{\left(2\right)}\right)^{c}}{\sum}f_{N}^{+}\left(\mathbf{x_{1}}\right)f_{N}^{+}\left(\mathbf{x_{2}}\right)
\]

\begin{equation}
\cdot\exp\left(-\frac{\beta^{2}}{4}+o\left(1\right)+\left(\frac{\beta}{2N}+o\left(1\right)\right)\left(s_{N}^{2}\left(\mathbf{x}_{1}\right)+s_{N}^{2}\left(\mathbf{x}_{2}\right)\right)+\frac{o\left(1\right)}{2N}\left(\stackrel[i]{N}{\sum}x_{i,1}x_{i,2}\right)^{2}\right),\label{eq:typical-atypical-expansion-square}
\end{equation}

by lemma \ref{thm:Kabluchko-beta-1}. Because $\beta<1$, there exists
a $\delta\in\mathbb{R}^{+}$ such that for large enough $N$:

\[
\left(\frac{\beta}{2N}+o\left(1\right)\right)\left(s_{N}^{2}\left(\mathbf{x}_{1}\right)+s_{N}^{2}\left(\mathbf{x}_{2}\right)\right)+\frac{o\left(1\right)}{2N}\left(\stackrel[i]{N}{\sum}x_{i,1}x_{i,2}\right)^{2}
\]

\begin{equation}
\leq\frac{\left(1-\delta\right)}{2N}\left(s_{N}^{2}\left(\mathbf{x}_{1}\right)+s_{N}^{2}\left(\mathbf{x}_{2}\right)+\left(\stackrel[i]{N}{\sum}x_{i,1}x_{i,2}\right)^{2}\right).
\end{equation}

Using this inequality, we have, for the atypical spin configurations
and large enough $N$:

\[
\underset{\mathbf{\left(x_{1},x_{2}\right)}\in\left(\mathcal{X}_{T,N}^{\left(2\right)}\right)^{c}}{\sum}f_{N}^{+}\left(\mathbf{x_{1}}\right)f_{N}^{+}\left(\mathbf{x_{2}}\right)
\]

\begin{equation}
\cdot\exp\left(-\frac{\beta^{2}}{4}+o\left(1\right)+\left(\frac{\beta}{2N}+o\left(1\right)\right)\left(s_{N}^{2}\left(\mathbf{x}_{1}\right)+s_{N}^{2}\left(\mathbf{x}_{2}\right)\right)+\frac{o\left(1\right)}{2N}\left(\stackrel[i]{N}{\sum}x_{i,1}x_{i,2}\right)^{2}\right)
\end{equation}

\begin{equation}
\leq\left(2^{N}o\left(1\right)\right)^{2},
\end{equation}

by lemma \ref{lem:exp-delta-3D-bounded}.

Inserting this into eq. (\ref{eq:typical-atypical-expansion-square})
yields:

\begin{equation}
E_{\varepsilon,N}\left(\left(\frac{E_{N}^{\left(\mu,BG\right)}\left(f_{N}^{+}\right)}{\cosh\left(\frac{\beta}{2Np}\right)\exp\left(\stackrel[i,j]{N}{\sum}\varepsilon_{N,i,j}\right)}\right)^{2}\right)
\end{equation}

\[
\leq\exp\left(\left|c_{N,2}\right|\right)e{}^{-\frac{\beta^{2}}{4}}\underset{\mathbf{\left(x_{1},x_{2}\right)}\in\mathcal{X}_{T,N}^{2}}{\sum}f_{N}^{+}\left(\mathbf{x_{1}}\right)f_{N}^{+}\left(\mathbf{x_{2}}\right)\exp\left(\frac{\beta}{2}\left(\frac{s_{N}^{2}\left(\mathbf{x}_{1}\right)+s_{N}^{2}\left(\mathbf{x}_{2}\right)}{N}\right)\right)
\]

\begin{equation}
+\left(2^{N}o\left(1\right)\right)^{2}
\end{equation}

\begin{equation}
\leq\exp\left(\left|c_{N,2}\right|\right)e{}^{-\frac{\beta^{2}}{4}}\left(\underset{\mathbf{x}\in\mathcal{X}_{N}}{\sum}f_{N}^{+}\left(\mathbf{x}\right)e^{\frac{\beta}{2N}s_{N}^{2}\left(\mathbf{x}\right)}\right)^{2}+\left(2^{N}o\left(1\right)\right)^{2}\label{eq:square expectation}
\end{equation}

\begin{equation}
=\exp\left(\left|c_{N,2}\right|\right)\left(e^{-\frac{\beta^{2}}{8}}E_{N}^{\left(\mu,CW\right)}\left(f_{N}^{+}\right)\right)^{2}\left(1+\left(\frac{o\left(1\right)2^{N}}{e^{-\frac{\beta^{2}}{8}}E_{N}^{\left(P,CW\right)}\left(f_{N}^{+}\right)Z_{N}^{\left(CW\right)}}\right)^{2}\right)
\end{equation}

\begin{equation}
\leq\exp\left(\left|c_{N,2}\right|\right)\left(e^{-\frac{\beta^{2}}{8}}E_{N}^{\left(\mu,CW\right)}\left(f_{N}^{+}\right)\right)^{2}\left(1+o\left(1\right)\right)\label{eq:O-division}
\end{equation}

(for large enough $N$, using proof assumptions and lemma \ref{lem:Z(CW)-order})

\begin{equation}
=a_{N,2}\left(e^{-\frac{\beta^{2}}{8}}E_{N}^{\left(\mu,CW\right)}\left(f_{N}^{+}\right)\right)^{2},
\end{equation}

with 

\begin{equation}
a_{N,2}\coloneqq\exp\left(\left|c_{N,2}\right|\right)\left(1+o\left(1\right)\right)\longrightarrow1,\quad Np\longrightarrow\infty.
\end{equation}
\end{proof}

\subsubsection{Proof of proposition \ref{lem:moment convergence-Np}, part 1}
\begin{proof}
We consider fixed $\delta>0$ and fixed $N$. By Markov's inequality:

\begin{equation}
P_{\varepsilon}\left(\left\{ \omega\in\varOmega_{\varepsilon}:\left|R_{N}\left(f_{N}^{+}\right)-1\right|>\delta\right\} \right)
\end{equation}

\begin{equation}
\leq\delta^{-2}E_{\varepsilon,N}\left(\left(R_{N}\left(f_{N}^{+}\right)-1\right)^{2}\right)
\end{equation}

\begin{equation}
\leq\delta^{-2}\left(a_{N,2}-2a_{N,1}+1\right)\longrightarrow0,
\end{equation}

as $Np\longrightarrow\infty$, using lemmata \ref{lem:exp-2} and
\ref{lem:sq-exp-2}.
\end{proof}

\subsubsection{Proof of proposition \ref{lem:moment convergence-Np}, part 2}
\begin{proof}
We consider fixed $\delta>0$ and fixed $N$. By Markov's inequality:
\begin{equation}
P_{\varepsilon}\left(\left\{ \omega\in\varOmega_{\varepsilon}:\left|T_{N}\left(f_{N}^{+}\right)\right|>\delta\right\} \right)
\end{equation}

\begin{equation}
\leq\delta^{-2}E_{\varepsilon,N}\left(T_{N}^{2}\left(f_{N}^{+}\right)\right)
\end{equation}

\begin{equation}
=\delta^{-2}E_{\varepsilon,N}\left(\left(R_{N}\left(f_{N}^{+}\right)\frac{E_{N}^{\left(\mu,CW\right)}\left(f_{N}^{+}\right)}{Z_{N}^{\left(\mu,CW\right)}}\right)^{2}\right)
\end{equation}

\begin{equation}
=\delta^{-2}\left(E_{N}^{\left(P,CW\right)}\left(f_{N}^{+}\right)\right)^{2}E_{\varepsilon,N}\left(R_{N}^{2}\left(f_{N}^{+}\right)\right)
\end{equation}

\begin{equation}
\leq\delta^{-2}\left(E_{N}^{\left(P,CW\right)}\left(f_{N}^{+}\right)\right)^{2}a_{N,2}\longrightarrow0,
\end{equation}

as $Np\longrightarrow\infty$, using lemma \ref{lem:sq-exp-2} and
proof assumptions.
\end{proof}

\subsection{Convergence of bounded integrals}
\begin{prop}
\label{lem:moment convergence}Assume that $\underset{N\rightarrow\infty}{\lim}E_{N}^{\left(P,CW\right)}\left(f_{N}^{+}\right)<\infty$
and $\underset{N\rightarrow\infty}{\lim}E_{N}^{\left(P,CW\right)}\left(f_{N}^{-}\right)<\infty$
exist. Then:

\[
E_{N}^{\left(P,BG\right)}\left(f_{N}\right)\overset{\mathbb{P}_{\varepsilon}}{\longrightarrow}\underset{N\rightarrow\infty}{\lim}E_{N}^{\left(P,CW\right)}\left(f_{N}\right).
\]
\end{prop}
\begin{proof}
First, we consider the case $\underset{N\rightarrow\infty}{\lim}E_{N}^{\left(P,CW\right)}\left(f_{N}^{+}\right)>0$
and $\underset{N\rightarrow\infty}{\lim}E_{N}^{\left(P,CW\right)}\left(f_{N}^{-}\right)>0$.
By lemmata \ref{lem:moment convergence-Np} and \ref{prop:stochastic_convergence_sure_subsequences},
every subsequence of $\mathbb{N}$ has subsequences $\left(n_{i}^{\left(1\right)}\right)_{i\in\mathbb{N}}$,
$\left(n_{i}^{\left(2\right)}\right)_{i\in\mathbb{N}}$, $\left(n_{i}^{\left(3\right)}\right)_{i\in\mathbb{N}}$
such that $R_{n_{i}^{\left(1\right)}}\left(1\right)\longrightarrow1$,
$R_{n_{i}^{\left(2\right)}}\left(f_{N}^{+}\right)\longrightarrow1$
and $R_{n_{i}^{\left(3\right)}}\left(f_{N}^{-}\right)\longrightarrow1$
almost surely. Now let $\left(n_{i}\right)_{i\in\mathbb{N}}\coloneqq\left(n_{i}^{\left(1\right)}\right)_{i\in\mathbb{N}}\cap\left(n_{i}^{\left(2\right)}\right)_{i\in\mathbb{N}}\cap\left(n_{i}^{\left(3\right)}\right)_{i\in\mathbb{N}}\neq\textrm{Ø}$
(by lemma \ref{prop:stochastic_convergence_sure_subsequences}). It
is easy to see that almost surely,
\begin{equation}
\underset{i\rightarrow\infty}{\lim}E_{n_{i}}^{\left(P,BG\right)}\left(f_{n_{i}}^{+}\right)=\underset{i\rightarrow\infty}{\lim}E_{n_{i}}^{\left(P,CW\right)}\left(f_{n_{i}}^{+}\right)=\underset{N\rightarrow\infty}{\lim}E_{N}^{\left(P,CW\right)}\left(f_{N}^{+}\right),
\end{equation}

similarly for $f_{N}^{-}$, and by additivity of the expectation,
also 
\[
\underset{i\rightarrow\infty}{\lim}E_{n_{i}}^{\left(P,BG\right)}\left(f_{n_{i}}\right)=\underset{N\rightarrow\infty}{\lim}E_{N}^{\left(P,CW\right)}\left(f_{N}\right).
\]

By lemma \ref{prop:stochastic_convergence_sure_subsequences}, it
follows that

\[
E_{N}^{\left(P,BG\right)}\left(f_{N}\right)\overset{\mathbb{P}_{\varepsilon}}{\longrightarrow}\underset{N\rightarrow\infty}{\lim}E_{N}^{\left(P,CW\right)}\left(f_{N}\right).
\]

If $\underset{i\rightarrow\infty}{\lim}E_{N}^{\left(P,CW\right)}\left(f_{N}^{+}\right)=0$,
or $\underset{i\rightarrow\infty}{\lim}E_{N}^{\left(P,CW\right)}\left(f_{N}^{+}\right)=0$,
or both, the same result holds.
\end{proof}

\subsection{Convergence in distribution (theorem \ref{thm:weak-convergence-probability-Rd})}
\begin{proof}
We first consider $d=1$. Let $h\in C_{b}\left(\mathbb{R}\right)$
and $f_{N}\left(\mathbf{x}\right)\coloneqq h\circ Y_{N}\left(\mathbf{x}\right)$.
We have: $f_{N}^{+}=h^{+}\circ Y_{N}$ and $f_{N}^{-}=h^{-}\circ Y_{N}$.
Because $f^{+}$ and $f^{-}$ are bounded and by theorem assumptions: 

\begin{equation}
\underset{N\rightarrow\infty}{\lim}E_{N}^{\left(P,CW\right)}\left(f_{N}^{+}\right)<\infty,\quad\underset{N\rightarrow\infty}{\lim}E_{N}^{\left(P,CW\right)}\left(f_{N}^{-}\right)<\infty
\end{equation}

both exist. Using proposition \ref{lem:moment convergence}, it follows
that 

\[
\int_{\mathbb{R}}h\left(x\right)d\left(\mathbb{P}_{N}^{\left(BG\right)}\circ Y_{N}^{-1}\right)=\underset{\mathbf{x}\in\mathcal{X}_{N}}{\sum}h\circ Y_{N}\left(\mathbf{x}\right)\mathbb{P}_{N}^{\left(BG\right)}\left(\mathbf{x}\right)
\]

\[
=E_{N}^{\left(P,BG\right)}\left(f_{N}\right)\overset{\mathbb{P}_{\varepsilon}}{\longrightarrow}\underset{N\rightarrow\infty}{\lim}E_{N}^{\left(P,CW\right)}\left(f_{N}\right)=\int_{\mathbb{R}}h\left(x\right)d\mathbb{P}^{\left(CW\right)}.
\]

This holds for any $h\in C_{b}\left(\mathbb{R}\right)$. We note that
the mapping $\Omega_{\varepsilon}\mapsto\left(\mathbb{P}_{N}^{\left(BG\right)}\circ Y_{N}^{-1}\right)$
is a random probability measure and $\mathbb{R}^{d}$ is a Polish
space. With this and \cite{Berti_2006}, it follows that 

\[
d_{L}\left(\mathbb{P}_{N}^{\left(BG\right)}\circ Y_{N}^{-1},\mathbb{P}^{\left(CW\right)}\right)\overset{\mathbb{P}_{\varepsilon}}{\longrightarrow}0.
\]

For $d>1$, we use $h\in C_{b}\left(\mathbb{R}^{d}\right)$ and replace
$h\circ Y_{N}\left(\mathbf{x}\right)$ with $h\circ\left(Y_{N}^{\left(1\right)},Y_{N}^{\left(2\right)},..,Y_{N}^{\left(d\right)}\right)\left(\mathbf{x}\right)$.
\end{proof}

\section{Statements and declarations}

\subsection*{Conflict of interests}

The author declares no competing interests.

\subsection*{Data availability }

There are no data associated with this manuscript.

\bibliographystyle{plain}
\bibliography{BACH_random_graphs}

\end{document}